\documentclass[12pt, a4paper]{article}
\usepackage{amsmath,amssymb,amsthm,amsfonts,latexsym,graphicx,subfigure}
\usepackage[all,import]{xy}
\usepackage[numbers,sort&compress]{natbib}
\usepackage{indentfirst}
\usepackage{fancyhdr}
\usepackage{bbm}
\usepackage{accents,cases}
\usepackage{multirow}
\usepackage[lined, ruled, linesnumbered, longend]{algorithm2e}

\usepackage{array}
\newcommand{\PreserveBackslash}[1]{\let\temp=\\#1\let\\=\temp}
\newcolumntype{C}[1]{>{\PreserveBackslash\centering}p{#1}}
\newcolumntype{R}[1]{>{\PreserveBackslash\raggedleft}p{#1}}
\newcolumntype{L}[1]{>{\PreserveBackslash\raggedright}p{#1}}

\DeclareMathOperator*{\argmin}{\ensuremath{arg\,min}}

\DeclareMathOperator*{\median}{\ensuremath{median}}
\DeclareMathOperator*{\sgn}{\ensuremath{Sgn}}
\DeclareMathOperator*{\card}{\ensuremath{Card}}
\DeclareMathOperator*{\sign}{\ensuremath{sign}}
\DeclareMathOperator*{\cone}{\ensuremath{cone}}
\DeclareMathOperator*{\cen}{\ensuremath{center}}
\DeclareMathOperator*{\vol}{\ensuremath{vol}}

\makeatletter
\def\wbar{\accentset{{\cc@style\underline{\mskip8mu}}}}

\makeatother

\renewcommand{\vec}[1]{\mbox{\boldmath \small $#1$}}

\newcommand{\diag}{\mathrm{diag}}

\newcommand{\videpost}{{vide post}{}}

\newcommand{\ie}{{i.e.}{~}}
\newcommand{\eg}{{e.g.}{~}}
\newcommand{\vs}{{v.s.}{~}}

\numberwithin{equation}{section}

\theoremstyle{plain}
\newtheorem{theorem}{Theorem}

\newtheorem{defn}{Definition}

\newtheorem{lemma}{Lemma}
\newtheorem{remark}{Remark}

\newtheorem{example}{Example}

\begin{document}
\bibliographystyle{unsrt} 
\title{The $1$-Laplacian Cheeger Cut: Theory and Algorithms}
\author{K.C. Chang\footnotemark[2],
\and Sihong Shao\footnotemark[2],
\and Dong Zhang\footnotemark[2]}
\renewcommand{\thefootnote}{\fnsymbol{footnote}}
\footnotetext[2]{LMAM and School of Mathematical Sciences, Peking University, Beijing 100871, China.}
\date{}
\maketitle

\begin{abstract}
This paper presents a detailed review of both theory and algorithms for the Cheeger cut based on the graph $1$-Laplacian. In virtue of the cell structure of the feasible set, we propose a cell descend (CD) framework for achieving the Cheeger cut. While plugging the relaxation to guarantee
the decrease of the objective value in the feasible set, from which both the inverse power (IP) method and the steepest descent (SD) method can also be recovered, we are able to get two specified CD methods. A compare study of all these methods are conducted on several typical graphs.
\end{abstract}

\section{Introduction}

Graph cut, partitioning the vertices of a graph into two or more disjoint subsets,
is a fundamental problem in graph theory \cite{Chung1997}.
It is a very powerful tool in data clustering with wide applications ranging from statistics, computer learning, image processing,  biology to social sciences \cite{JainMurtyFlynn1999}.
There exist several kinds of balanced graph cut \cite{ShiMalik2000,HeinSetzer2011,BuhlerRangapuramSetzerHein2013}.
The Cheeger cut \cite{Cheeger1970}, which has recently been shown to provide excellent classification results~\cite{SzlamBresson2010,HeinBuhler2010,BressonLaurentUminskyBrecht2013}, is one of them and its definition is as follows. Let $G=(V,E)$ denote a undirected and unweighted graph with vertex set $V=\{1,2,\cdots,n\}$ and edge set $E$. Each edge $e\in E$ is a pair of vertices
$\{i,j\}$. For any vertex $i$, the degree of $i$, denoted by $d_i$, is defined to be the number of edges passing through $i$. Let $S$ and $T$ be two nonempty subsets of $V$ and use
\[
E(S,T)=\left\{\{i,j\}\in E:i\in S,j\in T\right\}
\]
to denote the set of edges between $S$ and $T$. The edge boundary of $S$ is $\partial S = E(S,S^c)$ ($S^c$ is the complement of $S$ in $V$) and
the volume of $S$ is defined to be $\vol(S): =\sum_{i\in S}d_i.$ The number
\[
h(G)=\min\limits_{S\subset V,S\not\in\{\emptyset,V\}} \frac{|\partial S|}{\min\{\vol(S),\vol(S^c)\}}
\]
is called the {\sl{Cheeger constant}}, and a partition $(S,S^c)$ of $V$ is called a {\sl{Cheeger cut}} of $G$ if
\[
\frac{|\partial S|}{\min\{\vol(S),\vol(S^c)\}}=h(G),
\]
where $|\partial S|$ is the cardinality of the set $\partial S$.

However, solving analytically the Cheeger cut problem is combinatorially NP-hard \cite{SzlamBresson2010}. Approximate solutions are required.
The most well-known approach to approximate the Cheeger cut solutions is the spectral clustering method, which relaxes the original discrete combination optimization problem into a continuous function optimization problem
through the graph Laplacian \cite{Luxburg2007}. The standard graph Laplacian (\ie the $2$-Laplacian) is defined as $L:=D-A$, where $D=\diag(d_1,\cdots,d_n)$ is a diagonal matrix and $A$ the adjacency matrix of $G$. According to the linear spectral graph thoery, the eigenvalues of $L$ satisfy $0=\lambda_1\le \lambda_2\le \cdots\le \lambda_n\le 2$ and the second eigenvalue $\lambda_2$ can be used to bound the Cheeger constant as follows
\begin{equation}
\label{eq:mu_2&lambda_2}
\frac{\lambda_2}{2}\leq h(G) \leq \sqrt{2\lambda_2},
\end{equation}
which is nothing but the Cheeger inequality \cite{Chung1997}. Furthermore,
the corresponding second eigenvector is also used to approximate the Cheeger cut, \ie the $2$-spectral clustering or the $\ell^2$ relaxation.
It should be noted that this second eigenvector is not the Cheeger cut, but only an approximation \cite{Luxburg2007}.

In order to achieve a better cut than the $\ell^2$ relaxation, a spectral clustering based on the graph $p$-Laplacian defined by
$$(\Delta_p\vec x)_i=\sum_{j\sim i} |x_i-x_j|^{p-1}\sign(x_i-x_j)$$
with small $p\in(1,2)$ was proposed, in view of the fact that
the cut by threshold the second eigenvector of the graph $p$-Laplacian
tends to the Cheeger cut as $p\rightarrow 1^+$~\cite{BuhlerHein2009}.
Here $j\sim i$ denotes vertex $j$ is adjacent to vertex $i$, $\sum_{j\sim i}$ means the summation is with respect to all vertices adjacent to vertex $i$,
and $\sign(t)$ is the standard sign function which
equals to $1$ if $t>0$, $0$ if $t=0$, and $-1$ if $t<0$.
The resulting $\ell^p$ relaxation
\begin{equation}
\label{eq:p_lap}
\frac{\sum_{i\sim j}|x_i-x_j|^p}{\sum_{i=1}^nd_i|x_i|^p}
\end{equation}
is differentiable but nonconvex, so that standard Newton-like methods can be applied, but only local minimizers are obtained. In actual calculations, multiple runs with random initializations are taken to approximate the global minimizer.

All above mentioned $p$-spectral clustering for any $p\in(1,2]$ are {\sl{indirect}} methods.

Since the second (the first non-zero) eigenvalue of $1$-Laplacian (see Definition~\ref{defn:1_lap}) for connected graphs equals to the Cheeger constant, and the corresponding eigenvectors provide exact solutions of the Cheeger cut problem~\cite{SzlamBresson2010,Chang2014}. We study the numerical solution of the second eigenvector of the graph $1$-Laplacian.

However, the $\ell^1$ nonlinear eigenvalue problem (the corresponding object function is obtained by setting $p=1$ in Eq.~\eqref{eq:p_lap}) is not only nonconvex but also nondifferentiable. Three types of algorithms have been proposed to minimize the $1$-spectral clustering problem. They are: the Split-Bregman like ratio minimization algorithm~\cite{SzlamBresson2010},
the inverse power (IP) method~\cite{HeinBuhler2010},
and the steepest descent (SD) algorithm~\cite{BressonLaurentUminskyBrecht2012}.
Unfortunately, all these methods fail to give global minimizers.

Motivated by a recently developed nonlinear spectral graph theory of $1$-Laplacian~\cite{Chang2014} by the first author of this paper, we propose a cell descend (CD) algorithm framework. The main idea is based on the fact that the feasible set consists of a collection of cells, and the objective function on
each cell is convex. At each step, one obtains the minimum on the initial cell, and finding out a suitable direction transfer to a new cell by descending the value of the objective function.
This CD algorithm framework combines the advantages of both the original discrete combination optimization
and its equivalent continuous optimization.

Preliminary numerical results on several typical graphs
demonstrate that the proposed CD algorithm framework could provide better cut than the IP and SD methods. In order to provide a solid theoretical foundation for algorithms, we also present a unified framework for designing algorithms for the $1$-Laplacian Cheeger cut problem, in which both IP and SD algorithms can be easily recovered but from a totally different angle.


The paper is organized as follows. In the rest of this section we list the notations, definitions and some known results.
 A brief review of the spectral theory of $1$-Laplacian and Cheeger cut is in Section~\ref{sec:theory}. A new CD framework for solving the Cheeger cut problem is proposed in Section~\ref{sec:alg}. The corresponding
numerical experiments on several typical graphs are presented with discussions in Section~\ref{sec:num}. The conclusion with a few remarks is presented in Section~\ref{sec:con}.

\subsection{Notations and definitions}

In the following, we list some notations which will be used in this paper.
\begin{itemize}
\item  To the edge $e\in E$, we assign an orientation, let $i$  be the head, and $j$ be the tail, they are denoted by $i= e_h,$ and $j= e_t$ respectively. Under this orientation, one defines the incidence matrix $B=(b_{ei})_{l\times n}$ where
\begin{equation}
b_{ei}=
\begin{cases}
1,  & \text{if} \;\;\; i=e_h,\\
-1, & \text{if} \;\;\; i=e_t,\\
0,  & \text{if} \;\;\; i\notin e,
\end{cases}
\end{equation}
with $e\in E$, $i\in V$ and $l$ being the number of edges in $E$.


\item For any $\vec x\in \mathbb{R}^n$,
we denote
\begin{align}
D_+(\vec x) &=\{i\in V|x_i>0\}, \quad
D_0(\vec x) =\{i\in V|x_i=0\}, \quad
D_-(\vec x) =\{i\in V|x_i<0\}, \\
\delta_\pm(\vec x) &=\sum_{i\in D_\pm(\vec x)}d_i, \quad
\delta_0(\vec x) =\sum_{i\in D_0(\vec x)}d_i, \quad \|\vec x\|_2 = \left(\sum_{i=1}^n x_i^2\right)^\frac12,\\
I(\vec x) &= \frac12 \sum_{i,j=1}^n a_{ij}|x_i-x_j|, \quad
\|\vec x\|_{1,d}  =\sum_{i=1}^nd_i|x_i|, \quad
F(\vec x) =\frac{I(\vec x)}{\|\vec x\|_{1,d}}. \label{eq:I1F}
\end{align}
\item Let
\begin{align}
X &=\{\vec x\in \mathbb{R}^n: \|\vec x\|_{1,d}=1\}, \label{eq:X}\\
\pi &=\{\vec x\in X: |\delta_+(\vec x)-\delta_-(\vec x)|\le \delta_0(\vec x)\}. \label{eq:pi_1}
\end{align}
\item For given $S\subset \mathbb{R}^n$,
define
\begin{equation}
\cone(S)=\{k\vec x:\vec x\in S,k>0\}.
\end{equation}
\item Given a subset $T\subset X$, let $\vec x, \vec y \in T$, we say that $\vec x$ is equivalent to $\vec y$ in $T$, denoted by $\vec x\simeq \vec y$ in $T$, if there is a path $\gamma$ connecting $\vec x$ and $\vec y$ in $T$, \ie $\exists$ a continuous $\gamma: [0, 1]\to T$ such that $\gamma(0)=\vec x$, $\gamma(1)=\vec y$.
\end{itemize}

\begin{defn}
\label{defn:1_lap}
Given a graph $G=(V,E)$, for any $\vec x\in\mathbb{R}^n$, the set valued map:
$$\Delta_1: \vec x \to \{B^T \vec z | \vec z: E\to \mathbb{R}^1\; \text{is a}\;\mathbb{R}^l\;\text{vector satisfying}\;z_e(B\vec x)_e=|(B\vec x)_e|,\;\forall\;e\in E\}$$
is called the graph $1$-Laplacian on $G$.
Namely, it can be rewritten as
\begin{equation}\label{eq:l-Lap}
\Delta_1 \vec x = B^T \sgn (B\vec x),
\end{equation}
where $\sgn: \mathbb{R}^n\to (2^\mathbb{R})^n$ is a set valued mapping:
\[
\sgn(\vec y) = (\sgn(y_1),\sgn(y_2),\cdots,\sgn(y_n)), \quad \forall \vec y=(y_1,y_2,\cdots,y_n)\in\mathbb{R}^n,
\]
and
\[
\sgn(t)=
\begin{cases}
1,  & \text{if} \;\;\; t>0,\\
-1, & \text{if} \;\;\; t<0,\\
[-1,1],  & \text{if} \;\;\; t=0.
\end{cases}
\]
\end{defn}

For convenience, we often write Eq.~\eqref{eq:l-Lap} in the coordinate form for $i=1,2,\cdots,n$:
\begin{equation}\label{eq:1-Lap-i}
\begin{aligned}
(\Delta_1 \vec x)_i &= (B^T \sgn(B\vec x))_i\\
&= \left\{\left.\sum_{j\sim i} z_{ij}(\vec x)\right|z_{ij}(\vec x)\in \sgn(x_i-x_j),\; z_{ji}(\vec x)=-z_{ij}(\vec x),\; \forall j\sim i\right\}.
\end{aligned}
\end{equation}
 Obviously, from Eq.~\eqref{eq:1-Lap-i}, the graph $1$-Laplacian $\Delta_1$ is a nonlinear set valued mapping, which is independent to the special choice of orientation.

\begin{defn} Given a graph $G=(V,E)$, a pair $(\mu, \vec x)\in \mathbb{R}^1\times X$ is called an eigenpair of the $1$-Laplacian $\Delta_1$ on $G$ if
\begin{equation}\label{eq:pair}
\mu D \sgn(\vec x)\bigcap \Delta_1 \vec x\neq \varnothing.
\end{equation}
In the coordinate form, Eq.~\eqref{eq:pair} is equivalent to the system: $\exists z_{ij}(\vec x)\in \sgn(x_i-x_j)$ satisfying $z_{ji}(\vec x)=-z_{ij}(\vec x),\forall j\sim i$, and
\begin{equation}
\sum_{j\sim i} z_{ij}(\vec x)\in \mu d_i \sgn(x_i),\quad i=1,2,\cdots, n.
\end{equation}
The set of all solutions of Eq.~\eqref{eq:pair} is denoted by $S(G)$.
\end{defn}

\begin{remark}\label{remark:cut}
According to Theorem~\ref{th:chang2014} (\videpost),
one can choose either $S=D_+(\vec x)$ or $S=D_-(\vec x)$ based on an eigenvector $\vec x$ with respect to the second eigenvalue to produce a Cheeger cut $(S,S^c)$.
\end{remark}

\begin{defn}
Given a graph $G=(V,E)$ and $\vec x\in \mathbb{R}^n$, let $\sigma:\{1,2,\cdots,n\}\to\{1,2,\cdots,n\} $ be a permutation satisfying
$x_{\sigma(1)}\le x_{\sigma(2)}\le \cdots \le x_{\sigma(n)}$.
Then there exists unique $k_0\in \{1,\cdots,n\}$ such that $$-\sum_{i=1}^n d_{\sigma(i)}<\cdots <\sum_{i=1}^{k_0-1}d_{\sigma(i)}-\sum_{i=k_0}^nd_{\sigma(i)}< 0\le \sum_{i=1}^{k_0}d_{\sigma(i)}-\sum_{i=k_0+1}^nd_{\sigma(i)} <\cdots<\sum_{i=1}^n d_{\sigma(i)} .$$
The set
\begin{equation*}
\begin{cases}
[x_{\sigma(k_0)},x_{\sigma(k_0+1)}], & \text{if}\;\sum_{i=1}^{k_0}d_{\sigma(i)}-\sum_{i=k_0+1}^n d_{\sigma(i)} =0,\\
\{ x_{\sigma(k_0)}\}, & \text{if}\; \sum_{i=1}^{k_0}d_{\sigma(i)}-\sum_{i=k_0+1}^n d_{\sigma(i)} >0,
\end{cases}
\end{equation*} is said to be the {median} of $\vec x$, denoted by $\median(\vec x)$.
\end{defn}

\begin{defn}
Let $m:\{1,\cdots,n\}\to \{-1,0,1\}$ be a nonzero map. 
The set
\begin{align*}
\triangle_m:=\{\vec x\in X: \sum_{i=1}^nd_im(i)x_i=1\text{ with }m(i)x_i> 0\text{ for }m(i)\ne 0\}
\end{align*}
is called a cell.
\end{defn}

We have several remarks about the cell as follows.
\begin{itemize}
\item $\triangle_m$ are piecewise linear manifolds. The center of gravity of $\triangle_m$ is
$$
\cen(\triangle_m) =  \frac{1}{\delta}\sum_{i=1}^n d_i m(i) \vec{e}_i.
$$
The dimension of $\triangle_m$ is
$$\dim \triangle_m =\card\{m(i):m(i)\ne 0\}-1,$$
and the total number of $k$-dimensional cells is $C_n^{k+1}2^{k+1}$ for $k=0,1,\cdots,n-1$.
\item It is evident that $\sign(\vec x):=(\sign(x_1), \cdots, \sign(x_n))=m$ if and only if $\vec x\in\triangle_m$. Consequently, the set of all cells having nonempty intersection with $\pi$ can be denoted by $\Pi$:
\begin{equation}
\Pi=\{\triangle_m:\exists \vec x\in \pi\; \text{such that}\; \sign(\vec x)=m\}.
\end{equation}

\item Hereafter, we sometimes use $[i \cdot m(i):m(i)\ne 0]$ to denote the cell $\triangle_m$ for simplicity.
\end{itemize}



\section{Theory of $1$-Laplacian Cheeger cut}
\label{sec:theory}


In this section, we present a detailed review of the theory of $1$-Laplacian Cheeger cut, including the spectrum of $\Delta_1$,
the property of the feasible set and the connection between
the Cheeger constant $h(G)$ and the second eigenvalue $\mu_2$ of $\Delta_1$.

\subsection{Spectrum of $\Delta_1$}


The function $I(\vec x)$ defined in Eq.~\eqref{eq:I1F} is Lipschitzian on $\mathbb{R}^n$. Let $\tilde{I}=I|_X$. A characterization of the subgradient vector field on $X$ has been studied in \cite{Chang2014}, by which one proves

\begin{theorem}[Theorem 4.11 in \cite{Chang2014}]
\label{th:S(G)=K}
 $\vec x\in S(G)$ if and only if $\vec x$ is a critical point of $\tilde{I}$.
\end{theorem}

\begin{remark}
It must be pointed out that,
the definition (see Definition 4.4 in \cite{Chang2014}) of the critical point of $\tilde{I}$ adopted in Theorem~\ref{th:S(G)=K} is different from the usual one which studies the critical point of $\frac{I(\cdot)}{\|\cdot\|_{1,d}}$ through the so-called Clarke derivative \cite{HeinBuhler2010}. In fact, the critical points of $\frac{I(\cdot)}{\|\cdot\|_{1,d}}$ in the sense of the Clarke derivative must be the eigenvectors of $\Delta_1$,
but the inverse is not true. To clarify this, we present an example in Appendix A.
\end{remark}

Let $K$ denote the set of critical points of $\tilde{I}$, and let $K_c$ denote the subset of $K$ with critical value $c$.

The Liusternik-Schnirelmann theory is extended to study the multiplicity of the critical points for the even function $\tilde{I}$. The notion of genus due to Krasnoselski is introduced, see for instance, \cite{Chang1985,Rabinowitz1986}.
Let $T\subset\mathbb{R}^n\backslash \{0\}$ be a symmetric set, \ie $-T=T$ satisfying $0\notin T$. An integer valued function, which is called the genus of $T$, $\gamma: T \to \mathbb{Z}^+$ is defined to be:
\begin{equation*}
\gamma(T) =
\begin{cases}
0, & \text{if}\; T=\varnothing,\\
\min\limits\{k\in\mathbb{Z}^+: \exists\; \text{odd continuous}\; h: T\to S^{k-1}\}, & \text{otherwise}.
\end{cases}
\end{equation*}
Obviously, the genus is a topological invariant.

Let us define
\begin{equation}
c_k=\inf_{\gamma(T)\ge k} \max_{\vec x\in T\subset X} I(\vec x),\quad k=1,2, \cdots n.
\end{equation}
It can be proved that these $c_k$ are critical values of $\tilde{I}$. One has
$$c_1\le c_2 \le \cdots \le c_n,$$
and if
\begin{equation}
c=c_{k+1}=\cdots =c_{k+l},\,\, 0\le k\le k+l \le n,
\end{equation}
then $\gamma(K_c)\ge l.$

A critical value $c$ is said of multiplicity $l$, if $\gamma(K_c)=l$.

\begin{theorem}[Theorem 4.10 in \cite{Chang2014}] There are at least $n$ critical points $\phi_k,\,\,k=1,2, \cdots, n$ of $\tilde{I}$ such that
$\phi_k\in K_{c_k}$. Moreover, counting multiplicity, $\tilde{I}$ has at least $n$ critical values.
\end{theorem}

\begin{theorem}[Corollary 5.5 in \cite{Chang2014}] If a graph $G$ consists of $r$ connected components $G_1, G_2, \cdots, G_r$, then the eigenvalue $\mu=0$ has multiplicity $r$, \ie all the eigenvectors with respect to $\mu=0$ form a critical set $K_0$ with $\gamma(K_0)=r$.
\end{theorem}

\begin{lemma}
Assume that $\vec x,\vec y \in X$ satisfy:
\begin{description}
\item(P1) sign preserving property: $x_i=0\Rightarrow y_i=0$; $x_i<0\Rightarrow y_i\le0$; $x_i>0\Rightarrow y_i\ge0$,
\item(P2) order preserving property: $x_i=x_j\Rightarrow y_i=y_j$; $x_i<x_j\Rightarrow y_i\le y_j$,
\end{description}
$\forall\, i,j\in \{1,2,\cdots,n\}$. If $(\mu,\vec x)$ is an eigenpair of $\Delta_1$, then $(\mu,t\vec x+(1-t)\vec y)$ is also an eigenpair of $\Delta_1,\,\forall\,t\in [0,1]$.
\end{lemma}

\begin{proof}
Since $(\mu,\vec x)$ is an eigenpair, we have: $\exists  z_{ij}(\vec x)\in \sgn(x_i-x_j)$ satisfying
\[
z_{ji}(\vec x)=-z_{ij}(\vec x),\,\forall\,i\sim j
\;\;
\text{and}
\;\;
\sum_{j\sim i} z_{ij}(\vec x)\in \mu d_i \sgn (x_i),\,\, i=1, \cdots, n.
\]
By (P1) and (P2),
we deduce that  $\sgn(x_i)\subset \sgn(y_i)$ and $\sgn(x_i-x_j)\subset \sgn(y_i-y_j)$.
If one takes $z_{ij}(\vec y)=z_{ij}(\vec x), \forall\,i,j\in \{1,2,\cdots,n\}$, then
$z_{ij}(\vec y)\in \sgn(y_i-y_j)$ also satisfies
\[
z_{ji}(\vec y)=-z_{ij}(\vec y),\,\forall\,i\sim j
\;\;
\text{and}
\;\;
\sum_{j\sim i} z_{ij}(\vec y)\in \mu d_i \sgn (y_i),\,\, i=1, \cdots, n.
\]
These mean $(\mu,\vec y)$ is an eigenpair, too. More generally, for any given $t\in[0,1]$, we let $\widetilde{\vec y}=t\vec x+(1-t)\vec y$. It is easy to check that $\vec x$ and $\vec y:=\widetilde{\vec y}$ satisfy (P1) and (P2).
In the same way, we have $(\mu,\widetilde{\vec y})$ is also an eigenpair. This completes the proof.
\end{proof}

A vector $\hat{\vec x}=(\hat{x}_1, \hat{x}_2,\cdots, \hat{x}_n)\in X$ is called a binary valued vector, if $\hat{\vec x}=c(a_1, \cdots, a_n)$ for some $c>0$, where $a_1, a_2, \cdots, a_n$ are either $1$ or $0$.
By the above lemma, it follows
\begin{theorem}\label{th:binary}
For any eigenpair $(\mu,\vec x)$ of $\Delta_1$, there exists a binary valued vector $\hat{\vec x}$  such that $(\mu,\hat{\vec x})$ is an eigenpair and $\vec x\simeq\hat{\vec x}$ in $S(G)\cap I^{-1}(\mu)$.
Consequently, $\exists A\subset V$ such that $\mu=\frac{|\partial A|}{\vol (A)}$.
\end{theorem}
\begin{proof}
We only need to prove the last conclusion. In fact, let $A=\{i\in V \,|\, \hat{\vec x}_i=1\}.$ From the system:
$$ \sum_{j\sim i}z_{ij}(\hat{\vec x})=\mu d_i \hat{x}_i,\,\,\,\,\,\,\forall\,i,$$
after summation on both sides, we have $\mu=\frac{|\partial A|}{\vol(A)}$.
\end{proof}

\begin{theorem}\label{properties}
The following facts are readily observed.
\begin{description}
\item(1) For any eigenvalue $\mu$ of $\Delta_1$, we have $ 0\le \mu \le 1$.
\item(2) The distance of two different eigenvalues are at least $\frac{4}{n^2(n-1)^2}$, where $n=|V|$. So there are only finite eigenvalues of $\Delta_1$.
\item(3) If $(\mu, \vec x)$ is an eigenpair of $\Delta_1$, then $ I(\vec x)=\mu$. Moreover, if $\mu\ne0$, then $0\in \sum^n_{i=1} d_i \sgn(x_i)$.
\item(4) If $\vec x\in X$ with $x_1=x_2= \cdots = x_n$, then $\vec x=\pm\frac{1}{\delta}(1, 1, \cdots, 1)$ is an eigenvector with eigenvalue $\mu=0$. Conversely, if $G$ is connected, then
the eigenvector $\vec x$ corresponding to the eigenvalue $\mu=0$ must be $\vec x=\pm\frac{1}{\delta}(1, 1, \cdots, 1)$.
\end{description}
\end{theorem}
\begin{proof}
We only prove here the fact (2). Other facts follow from~\cite{Chang2014}.

For given two different  critical values $\mu,\tilde{\mu}$, by Theorem~\ref{th:binary}, we have:
$\exists A, B\subset V$ such that $\mu=\frac{|\partial A|}{\vol (A)}$ and $\tilde{\mu}=\frac{|\partial B|}{\vol (B)}$. Accordingly, we obtain
$$|\mu-\tilde{\mu}| =
\left|\frac{|\partial A|}{\vol(A)}-\frac{|\partial B|}{\vol(B)}\right|=\frac{||\partial A|\vol(B)-|\partial B|\vol(A)|}{\vol(A)\vol(B)}\ge \frac{1}{\vol(A)\vol(B)}\ge \left(\frac{1}{C_n^2}\right)^2.$$
\end{proof}

Since the number of eigenvalues are finite, all eigenvalues can be ordered as follow:
$$ 0=\mu_1\le \mu_2\le \cdots,$$
On the other hand, the subset of critical values (at least $n$ if counting multiplicit\vec y) $\{c_k\}$, is ordered by the topological genus, i.e.,
$$0=c_1\le c_2 \le \cdots \le c_n.$$
According to Theorem 1, $\mu$ is a eigenvalue of $\Delta_1$ if and only if $\mu$ is a critical value of $\widetilde{I}(\vec x)$.
But we do not know whether these two sequences are identical.

\begin{theorem}\label{th:pi_key}
Any eigenvector $\vec x$ of $\Delta_1$ with eigenvalue $\mu \neq 0$ lies on $\pi$.
\end{theorem}
\begin{proof}
By Theorem \ref{properties}, we have $0\in \sum^n_{i=1} d_i \sgn(x_i)$. So there exist $\theta_i\in \sgn(x_i)$ such that $\sum^n_{i=1} d_i\theta_i=0$, $i=1,2,\cdots,n$. Then
$$|\delta_+(\vec x)-\delta_-(\vec x)|=\left|\sum_{i\in D_+(\vec x)}d_i-\sum_{i\in D_-(\vec x)}d_i\right|=\left|\sum_{i\in D_0(\vec x)}\theta_id_i\right|\le \delta_0(\vec x),$$ which means that $\vec x\in\pi$.
\end{proof}


\subsection{Properties of $\pi$}

According to Theorems~\ref{th:pi_key} and \ref{th:chang2014} (\videpost),
$\pi$ is the feasible set for searching the Cheeger cut and some useful properties of $\pi$ will be shown in this section. By simple observation, we have

\begin{theorem}\label{th:compactness}
$\pi$ is compact and $\pi=\bigcup\limits_{\triangle\in\Pi}\triangle$.
\end{theorem}

\begin{lemma}\label{lm:pi}
$\vec x\in\pi$ if and only if $\delta_+(\vec x)\le \frac{\delta}{2}$ and $\delta_-(\vec x)\le \frac{\delta}{2}$ with $\delta=\sum^n_{i=1} d_i$.
\end{lemma}
\begin{proof}
Without loss of generality, we assume $\delta_+(\vec x)\ge \delta_-(\vec x)$, then
\begin{align*}
\vec x\in\pi&\Leftrightarrow \delta_+(\vec x)-\delta_-(\vec x)\le \delta_0(\vec x)
\\&\Leftrightarrow \delta_+(\vec x)-\delta_-(\vec x)\le \delta-\delta_+(\vec x)-\delta_-(\vec x)
\\&\Leftrightarrow 2\delta_+(\vec x)\le \delta
\\&\Leftrightarrow\delta_+(\vec x)\le \frac{\delta}{2}\text{ and }\delta_-(\vec x)\le \frac{\delta}{2}.
\end{align*}
\end{proof}

\begin{theorem}\label{th:pi_connected}
$\pi$ is connected when $n\ge 3$.
\end{theorem}

\begin{proof}
For any $\vec x\in \pi$, $\delta_+(\vec x)\le \frac{\delta}{2}$, $\delta_-(\vec x)\le \frac{\delta}{2}$. Without loss of generality, we may assume $x_1\ne 0$. On one hand, consider
\[
\gamma(t)=\frac{(x_1,tx_2,\cdots,tx_n)}{d_1|x_1|+t\sum_{i=2}^n d_i|x_i|}, \quad t\in [0,1].
\]
It is easily checked that both $\delta_+(\gamma(t))\le \delta_+(\vec x)\le \frac{\delta}{2}$ and $\delta_-(\gamma(t))\le \delta_-(\vec x)\le \frac{\delta}{2}$ hold, then we have $\gamma([0,1])\subset \pi$.
This means that the path $\gamma$ connecting $\vec x$ and $(\frac{\sign(x_1)}{d_1},0,\cdots,0)$ lies in $\pi$,
\ie $x\simeq (\frac{\sign(x_1)}{d_1},0,\cdots,0)$.
 On the other hand, when $n\ge 3$, we can prove $(\pm\frac{1}{d_1},0,0,\cdots,0)\simeq (0,\pm \frac{1}{d_2},0,\cdots,0)$. Actually, for $(\frac{1}{d_1},0,0,\cdots,0)$ and $(0, -\frac{1}{d_2},0,\cdots,0)$, this can be readily verified by taking the following path
$$y(t)=\left(\frac{t}{d_1},\frac{t-1}{d_2},0,\cdots,0\right)\subset \pi,\,\,\,\forall\,t\in[0,1].$$
For $(\frac{1}{d_1},0,0,\cdots,0)$ and $(0, \frac{1}{d_2},0,\cdots,0)$, we then have
\[
(\frac{1}{d_1},0,0,\cdots,0) \simeq
(0,0,-\frac{1}{d_3},\cdots,0) \simeq
(0, \frac{1}{d_2},0,\cdots,0).
\]
Hence, for any $\vec y \in\pi$ with $y_i\neq 0$, we have $x\simeq (\frac{1}{d_1},0,0,\cdots,0)\simeq(0,\cdots,\frac{1}{d_i},\cdots,0)\simeq \vec y$.
\end{proof}

\begin{remark}
When $n=2$, it can be easily checked that $\pi$ consists of two disjoint connected subsets.
\end{remark}

\begin{lemma}
$\vec y \in\median(\vec x)$ if and only if $\vec y=\argmin\limits_{t\in \mathbb{R}}\sum_{i=1}^n d_i|t-x_i|$.
\end{lemma}
\begin{proof}
Let $f(t)=\sum_{i=1}^nd_i|t-x_i|$.

Without loss of  generality, we assume $x_1\le x_2\le \cdots \le x_n$ and $x_1< x_n$, then there exists $k\in \{1,2,\cdots,n-1\}$ such that $x_k<x_{k+1}$.

Assume $x_k\le t<t+h< x_{k+1}$, then we have $f(t+h)-f(t)=h\sum_{i=1}^kd_i-h\sum_{i=k+1}^nd_i$.

If $\sum_{i=1}^kd_i-\sum_{i=k+1}^nd_i\le 0$, then $f(t+h)-f(t)\le 0$, so $f$ is decreasing on $[x_k,x_{k+1}]$.

If $\sum_{i=1}^kd_i-\sum_{i=k+1}^nd_i\ge 0$, then $f(t+h)-f(t)\ge 0$, so $f$ is increasing on $[x_k,x_{k+1}]$.

Since there exists $k_0\in \{1,\cdots,n\}$ such that $$-\sum_{i=1}^n d_i<\cdots <\sum_{i=1}^{k_0-1}d_i-\sum_{i=k_0}^nd_i< 0\le \sum_{i=1}^{k_0}d_i-\sum_{i=k_0+1}^nd_i <\cdots<\sum_{i=1}^n d_i ,$$ we deduce that $f(t)$ is decreasing on $(-\infty,x_{k_0}]$ and increasing on $[x_{k_0+1},+\infty)$.

When $\sum_{i=1}^{k_0}d_i-\sum_{i=k_0+1}^nd_i =0$,
we have $\median(\vec x)=[x_{k_0},x_{k_0+1}]$, and
$f$ is constant on $\median(\vec x)$. Consequently, the set of minimal points of $f$ equals to $\median(\vec x)$.

When $\sum_{i=1}^{k_0}d_i-\sum_{i=k_0+1}^nd_i >0$,
we have $\median(\vec x)=\{x_{k_0}\}$ and $x_{k_0}$ is the only minimal point of $f$.
\end{proof}

\begin{theorem}\label{th:median}
For any $\vec x\in X$, we have
\[
0\in \median(\vec x)\Leftrightarrow \vec x\in\pi.
\]
In other word,
\begin{equation}\label{eq:pi_2}
\pi = \{\vec x\in X: 0\in \median(\vec x)\}.
\end{equation}
\end{theorem}
\begin{proof}
Suppose $x_1\le\cdots\le x_n$.

$\Rightarrow$:

If there exists $k$ such that $\sum_{i=1}^kd_i-\sum_{i=k+1}^nd_i=0$, then $\median(\vec x)=[x_k,x_{k+1}]$ and $D_-(\vec x)\subset \{1,2,\cdots,k\}$, $D_+(\vec x)\subset\{k+1,\cdots,n\} $. Then we have $\delta_-(\vec x)\le \sum_{i=1}^k d_i=\frac{\delta}{2}$ and $\delta_+(\vec x)\le \sum_{i=k+1}^n d_i=\frac{\delta}{2}$.  Otherwise, there exists $k$ such that $\sum_{i=1}^{k-1}d_i-\sum_{i=k}^nd_i <0<\sum_{i=1}^kd_i-\sum_{i=k+1}^nd_i$, then $\median(\vec x)=\{x_k\}$. Since $0\in \median(\vec x)$, $x_k=0$. Therefore, $D_-(\vec x)\subset \{1,2,\cdots,k-1\}$, $D_+(\vec x)\subset\{k+1,\cdots,n\} $. We still have $\delta_-(\vec x)\le \sum_{i=1}^{k-1} d_i<\frac{\delta}{2}$ and $\delta_+(\vec x)\le \sum_{i=k+1}^n d_i<\frac{\delta}{2}$. In either case, by Lemma~\ref{lm:pi}, we always have $\vec x\in\pi$.

$\Leftarrow$:

We suppose the contrary that $0\notin\median(\vec x)$

When $\min\limits \median(\vec x)>0$, let $x_k=\min\limits \median(\vec x)$,
then we have
$$\sum_{i=1}^{k-1}d_i-\sum_{i=k}^nd_i<0,$$
and $\delta_+(\vec x)\ge \sum_{i=k}^nd_i>\frac{\delta}{2}$ implying that $x\notin\pi$, which is a contradiction.

When $\max \median(\vec x)<0$, let $x_{k+1}=\min\limits \median(\vec x)$, then we have
$$0\le \sum_{i=1}^{k}d_i-\sum_{i=k+1}^nd_i < \sum_{i=1}^{k+1}d_i-\sum_{i=k+2}^nd_i,$$
and $\delta_-(\vec x)\ge \sum_{i=1}^{k+1}d_i>\frac{\delta}{2}$ implying that $\vec x\notin\pi$. Again, it is a contradiction.
\end{proof}

In summary, the compactness of the feasible set $\pi$ in Theorem~\ref{th:compactness} ensures the existence of a minimizer of the continuous function $I(\vec x)$ and the connectivity in Theorem~\ref{th:pi_connected} provides a possibility of finding the minimizer via a descend method. Theorem~\ref{th:median} shows that the projection of $\vec x\in X$ onto $\pi$ can be realized by a simple operation $x-\median(\vec x)$. In Section~\ref{sec:alg}, new cell descent algorithms for seeking the Cheeger cut will be designed according to the cell structure of $\pi$ (Theorem~\ref{th:compactness}).

\subsection{Connection between $h(G)$ and $\mu_2$}

Now let's study the relationship between Cheeger's constant and the eigenvalue of $\Delta_1$.  In \cite{HeinBuhler2010,Chang2014}, it is proved:
\begin{theorem}[Proposition 4.1 in \cite{HeinBuhler2010}; Theorem 5.15 in \cite{Chang2014}]
\label{th:chang2014}
We have
\begin{equation}\label{eq:mu=h}
\mu_2=\min\limits_{\vec x\in\pi}I(\vec x)=h(G).
\end{equation}
\end{theorem}

Combining Theorems~\ref{th:median} and~\ref{th:chang2014} leads to~\cite{Chung1997,SzlamBresson2010}
\begin{equation}\label{th:Chung_d}
h(G)=\min\limits_{x\ne0,\; 0\in\median(\vec x)}\frac{I(\vec x)}{\|\vec x\|_{1,d}}=\min\limits_{x\;\text{nonconstant} \\ \;\alpha\in\median(\vec x)}\frac{I(\vec x)}{\|x-\alpha\vec{1}\|_{1,d}}.
\end{equation}

Theorem~\ref{th:chang2014} means that we could minimize the original discrete graph cut problem through solving an equivalent continuous function optimization problem.
As mentioned in Remark~\ref{remark:cut},
based on corresponding eigenvector $\vec x$,
via $\vec x\simeq \hat{\vec x}$ (\ie finding the equivalent binary vector $\hat{x}$) by Theorem~\ref{th:binary},
we are able to produce a unique Cheeger cut $(A,A^c)$ by setting either $A=D_+(\vec x)$ or $A=D_-(\vec x)$.
On the other hand,
for any Cheeger cut $(A,A^c)$,
we can also determine at most two binary eigenvectors $\hat{x}=\vec{1}_A/\vol(A)$ or $\hat{x}=\vec{1}_{A^c}/\vol(A^c)$.
Moreover, starting from an eigenvector $\vec x$ (\ie a Cheeger cut),
Theorem~\ref{th:connect} below implies that
it is possible for us to obtain a new binary eigenvector (\ie a new Cheeger cut) by investigating the connectedness of resulting subgraphs from the pre-existing Cheeger cut.
Actually,
the connectedness of those subgraphs
is closely related to the nodal domains of eigenvectors~\cite{Chang2014}.





\begin{theorem}\label{th:connect}
Let $G=(V,E)$ be connected, the following statements hold.
\begin{description}
\item(1) If $(A,A^c)$ is a Cheeger cut and $A$ is a disjoint union of two nonempty sets $A_1$, $A_2$ satisfying $E(A_1,A_2)=\emptyset$, then $\vol(A)\le \vol(A^c)$, and $(A_1,A_1^c)$, $(A_2,A_2^c)$ are also Cheeger cuts with $\frac{|\partial A_i|}{\vol(A_i)}=h(G),i=1,2$.
\item(2) If $(A,A^c)$ is a Cheeger cut, then one of $A$, $A^c$ is connected.


\item(3) There exists a Cheeger cut $(A,A^c)$ such that $A$, $A^c$ are both connected.
\end{description}
\end{theorem}
\begin{proof}

(1) One can easily check that $\partial A$ is the disjoint union of $\partial A_1$ and $\partial A_2$, it follows $|\partial A|=|\partial A_1|+|\partial A_2|$ and $\vol(A)=\vol(A_1)+\vol(A_2)$.
 Therefore,
 \begin{equation}\label{eq:vol1}
 \frac{|\partial A|}{\vol(A)}=\frac{|\partial A_1|+|\partial A_2|}{\vol(A_1)+\vol(A_2)}\ge \min\limits\left\{\frac{|\partial A_1|}{\vol(A_1)},\frac{|\partial A_2|}{\vol(A_2)}\right\}.
 \end{equation}

Meanwhile, $|\partial A_1^c|=|\partial A_1|<|\partial A|=|\partial A^c|$ and $A_1\varsubsetneqq A$ implies $\vol(A^c)<\vol(A_1^c)$. Then
\begin{equation}\label{eq:vol2}
\frac{|\partial A_1^c|}{\vol(A_1^c)}<\frac{|\partial A^c|}{\vol(A^c)}.
\end{equation}

We claim:
$$\frac{|\partial A_1|}{\vol(A_1)}=\frac{|\partial A_2|}{\vol(A_2)}=\frac{|\partial A|}{\vol(A)}.$$
For otherwise, according to Eq.~\eqref{eq:vol1}, we may assume without loss of generality that
$$\frac{|\partial A_1|}{\vol(A_1)}
\leq \frac{|\partial A_2|}{\vol(A_2)}.$$
Then we have
$$\frac{|\partial A_1|}{\vol(A_1)}
<
\frac{|\partial A|}{\vol(A)}.$$
It follows from Eq.~\eqref{eq:vol2},
$$\max\left\{\frac{|\partial A_1|}{\vol(A_1)},\frac{|\partial A_1^c|}{\vol(A_1^c)}\right\}<\max\left\{\frac{|\partial A|}{\vol(A)},\frac{|\partial A^c|}{\vol(A^c)}\right\}=h(G).$$
This is a contradiction.

Now we turn to prove: $\vol(A)\le \vol(A^c)$. If not, \ie $\vol(A)> \vol(A^c)$, then combining with Eq.~\eqref{eq:vol2}
yields
$$\max\left\{\frac{|\partial A_1|}{\vol(A_1)},\frac{|\partial A_1^c|}{\vol(A_1^c)}\right\}= \max\left\{\frac{|\partial A|}{\vol(A)},\frac{|\partial A_1^c|}{\vol(A_1^c)}\right\}<\frac{|\partial A^c|}{\vol(A^c)}=h(G).$$
Again, it is a contradiction. Thus
$$\max\left\{\frac{|\partial A_1|}{\vol(A_1)},\frac{|\partial A_1^c|}{\vol(A_1^c)}\right\}= \max\left\{\frac{|\partial A|}{\vol(A)},\frac{|\partial A_1^c|}{\vol(A_1^c)}\right\}=
\frac{|\partial A|}{\vol(A)}=
\max\left\{\frac{|\partial A|}{\vol(A)},\frac{|\partial A^c|}{\vol(A^c)}\right\}=h(G),$$
\ie $(A_1,A_1^c)$ is a Cheeger cut.
Similarly, $(A_2,A_2^c)$ is also a Cheeger cut.

(2) Suppose the contrary. There is  a Cheeger cut $(A,A^c)$ such that $A=A_1\sqcup A_2$, $A^c=B_1\sqcup B_2$, where $A_1,A_2,B_1,B_2$ are four disjoint nonempty sets satisfying $E(A_1,A_2)=\emptyset$, $E(B_1,B_2)=\emptyset$. By (1), we know $\vol(A)=\vol(A^c)$. That is, $\vol(A_1)+\vol(A_2)=\vol(B_1)+\vol(B_2)$. The sketch  for their possible connection is shown below.
\begin{center}
\includegraphics{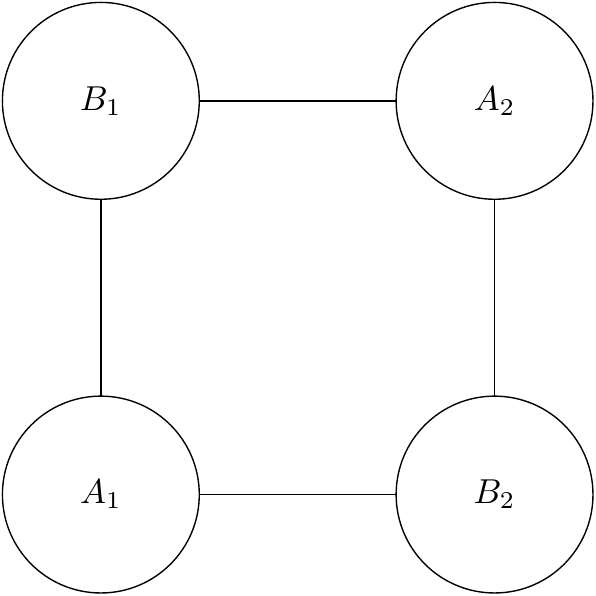}
\end{center}

According to (1), $(A_1,A_1^c)$, $(A_2,A_2^c)$, $(B_1,B_1^c)$, $(B_2,B_2^c)$ are all Cheeger cuts.
Since $\partial A_1 = E(A_1,B_1)\sqcup E(A_1,B_2)$, we have
$$|\partial A_1| = |E(A_1,B_1)|+ |E(A_1,B_2)|.$$ Similarly, we obtain
$$|\partial A_2| = |E(A_2,B_1)|+ |E(A_2,B_2)|,$$
$$|\partial B_1| = |E(B_1,A_1)|+ |E(B_1,A_2)|,$$
$$|\partial B_2| = |E(B_2,A_1)|+ |E(B_2,A_2)|.$$
The fact $\frac{|\partial A_1|}{\vol(A_1)}=h(G)$ leads to
$$|E(A_1,B_1)|+ |E(A_1,B_2)| = |\partial A_1| =h(G) \vol(A_1).$$
Analogously, we have
$$|E(B_1,A_1)|+ |E(B_1,A_2)|= h(G) \vol(B_1).$$

Since $G$ is connected, at least three of $E(A_1,B_1), E(A_1,B_2), E(A_2,B_2), E(A_2,B_1)$ are nonempty. Then we may assume that $|E(A_1,B_1)|>0$, $|E(A_2,B_2)|>0$.

Let $D=A_1\sqcup B_1$ and $D^c=A_2\sqcup B_2$. Then we have
$$|\partial D|=|E(A_1,B_2)|+|E(A_2,B_1)|,$$
and
$$\vol(D)=\vol(A_1)+\vol(B_1)=\frac{1}{h(G)}\left(2|E(A_1,B_1)|+ |E(A_1,B_2)|+|E(A_2,B_1)|\right).$$
Thus,
$$\frac{|\partial D|}{\vol(D)}=h(G)\frac{|E(A_1,B_2)|+|E(A_2,B_1)|}{2|E(A_1,B_1)|+ |E(A_1,B_2)|+|E(A_2,B_1)|}<h(G).$$
Similarly, $\frac{|\partial D^c|}{\vol(D^c)}<h(G)$, we obtain a contradiction.

(3) We suppose that $A$ is not connected, then by (2) we know that $A^c$ is connected. There are finite disjoint nonempty sets $A_1$, $A_2$, $\cdots$, $A_k$ such that $A=\bigsqcup_{i=1}^k A_i$ with $k\ge 2$. Each subgraph $(A_i,E(A_i))$ is connected. By noting
that $|E(A_i,A^c)|>0$, $i=1,2,\cdots, k$, we have $A_1^c=A^c\cup \bigsqcup_{i=2}^kA_i$ is connected.
That is, both $A_1$ and $A_1^c$ are connected. Further by (1), we know that $(A_1,A_1^c)$ is a Cheeger cut.
\end{proof}

\begin{example}
The Cheeger value of the following graph is $\frac{1}{5}$. Let $A_1=\{1,2,3\}$, $A_2=\{11,12,13\}$, $A=A_1\sqcup A_2$. It can be easily verified that $(A,A^c)$ is a Cheeger cut.
\begin{center}
\includegraphics{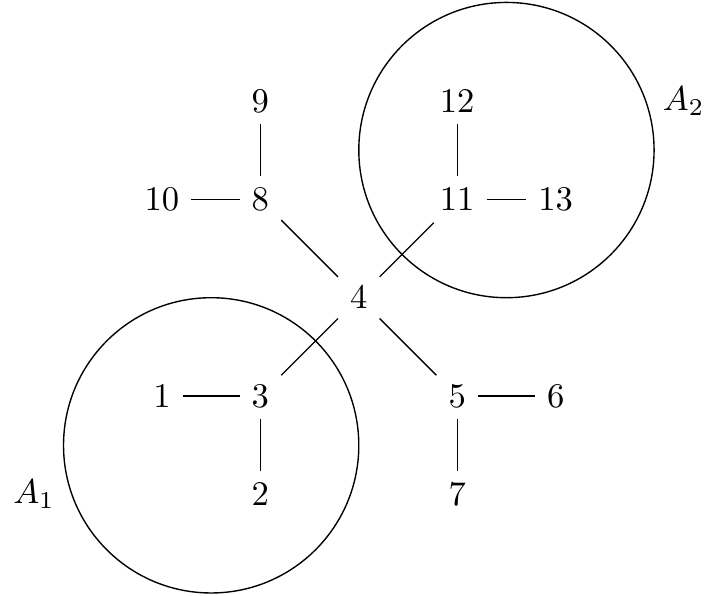}
\end{center}
It is seen that,  $A^c$ is connected but $A$ is not. One can also easily check that,  $\vol(A)=10< \vol(A^c)=14$; both $(A_1,A_1^c)$ and $(A_2,A_2^c)$ are Cheeger cuts; both $(A_1, E(A_1))$ and $(A_1^c, E(A_1^c))$ are connected subgraphs. Theorem~\ref{th:connect} is verified on this graph.
\end{example}

Theorem~\ref{th:chang2014} is the cornerstone of our noval cell descent algorithms delineated in Section~\ref{sec:alg}. Theorem~\ref{th:connect} gives us some hint for choosing an initial cut. That is, it is better to start from the initial cut each branch of which is connected.


\section{Algorithms for $1$-Laplacian Cheeger cut}
\label{sec:alg}



According to Theorems~\ref{th:compactness} and \ref{th:chang2014}, we have
\begin{equation}\label{eq:h=mu=cell}
h(G) = \mu_2=\min\limits_{\vec x\in\pi}I(\vec x) = \min\limits_{\Delta\in\Pi} \min\limits_{\vec x\in\bar{\Delta}} I(\vec x),
\end{equation}
from which there are two different approaches in treating the Cheeger cut problem. The first one is to regard it as a discrete combination optimization problem. By Theorem~\ref{th:compactness}, the minimizer of $I(\vec x)$ in $\pi$ exists, it can be found by solving a convex programming on each of the cells, and then by choosing the smallest one among them. However, the number of cells in $\Pi$ must be larger than $\sum_{k=0}^{[\frac{n-1}{2}]} C_n^{k+1}2^{k+1}=3^{[\frac{n+1}{2}]}-1$, which is reminiscent of NP-hardness of the Cheeger cut problem.
 The other one is to regard it as a continuous optimization problem on the feasible set $\pi$. Both IP (see Eq.~\eqref{eq:4step}) and SD (see Eqs.~\eqref{eq:4step-gsd} and \eqref{eq:sd}) methods have been used, but usually only local minimizers are obtained.

In this section, we propose a new CD framework in solving the combinatorial optimization problem. Since the objective function is convex on each cell, we produce a sequence of cells $\triangle_k$ in $\pi$, on which the minimizers $\{\vec x^k\}$ are decreasing.

In each step, we search a decreasing direction of $I$ in $\pi$ at the minimizer $\vec x^k$ in order to transfer the search from one cell to the other.

We are going to study the strategy for generating a decreasing direction.

\begin{lemma}\label{lm:x0}
Let $T\subset \mathbb{R}^n-\{0\}$ be a bounded closed set. We have
\[
\vec x^0=\argmin\limits_{\vec x\in T}\frac{I(\vec x)}{\|\vec x\|_{1,d}} \Leftrightarrow \vec x^0=\argmin\limits_{\vec x\in T} \{I(\vec x)-\lambda^0 \|\vec x\|_{1,d}\},
\]
where $\lambda^0=\frac{I(\vec x^0)}{\|\vec x^0\|_{1,d}}$.
\end{lemma}
\begin{proof}
$\Rightarrow$:

Since $\vec x^0=\argmin\limits_{\vec x\in T}\frac{I(\vec x)}{\|\vec x\|_{1,d}} $, we have
$$\frac{I(\vec x)}{\|\vec x\|_{1,d}}\ge \frac{I(\vec x^0)}{\|\vec x^0\|_{1,d}}=\lambda^0,\,\,\,\,\,\forall\, \vec x\in T .$$
Therefore, $I(\vec x)-\lambda^0\|\vec x\|_{1,d}\ge 0$. Now $I(\vec x^0)-\lambda^0\|\vec x^0\|_{1,d}=0$, \ie $\vec x^0=\argmin\limits_{\vec x\in T} \{I(\vec x)-\lambda^0 \|\vec x\|_{1,d}\}$.

$\Leftarrow$:

Since $\vec x^0=\argmin\limits_{\vec x\in T} \{I(\vec x)-\lambda^0 \|\vec x\|_{1,d}\}$, we have
$$I(\vec x)-\lambda^0 \|\vec x\|_{1,d}\ge I(\vec x^0)-\lambda^0\|\vec x^0\|_{1,d}=0,\,\,\,\,\, \forall\,\vec x\in T.$$
Therefore $\frac{I(\vec x)}{\|\vec x\|_{1,d}}\ge \lambda^0=\frac{I(\vec x^0)}{\|\vec x^0\|_{1,d}}$, then $\vec x^0=\argmin\limits_{\vec x\in T}\frac{I(\vec x)}{\|\vec x\|_{1,d}} $.
\end{proof}

\begin{lemma}\label{lm:x1}
Let $T\subset \mathbb{R}^n-\{0\}$  be a bounded closed set and $\vec x^0\in T$ be a given point. Assume
\[
 \vec x^1=\argmin\limits_{\vec x\in T} \{I(\vec x)-\lambda^0 \|\vec x\|_{1,d}\},
\]
then $\lambda^1\le \lambda^0$, where $\lambda^0=\frac{I(\vec x^0)}{\|\vec x^0\|_{1,d}}$
and $\lambda^1=\frac{I(\vec x^1)}{\|\vec x^1\|_{1,d}}$.
\end{lemma}
\begin{proof}
Since $\vec x^1=\argmin\limits_{\vec x\in T} \{I(\vec x)-\lambda^0 \|\vec x\|_{1,d}\}$, we have
$$I(\vec x)-\lambda^0 \|\vec x\|_{1,d}\ge I(\vec x^1)-\lambda^0\|\vec x^1\|_{1,d},\,\,\,\,\forall\, \vec x\in T.$$
Taking $x=\vec x^0$ in the above equation, we have
$$0=I(\vec x^0)-\lambda^0 \|\vec x^0\|_{1,d}\ge I(\vec x^1)-\lambda^0\|\vec x^1\|_{1,d}.$$

Therefore, $\lambda^0\ge\frac{I(\vec x^1)}{\|\vec x^1\|_{1,d}}=\lambda^1$.
\end{proof}

Based on Lemmas \ref{lm:x0} and~\ref{lm:x1}, we obtain

\begin{theorem}Let $T\subset \{\vec x\in \mathbb{R}^n-\{0\}:0\in\median(\vec x)\}$ be a bounded closed set containing $\pi$. The two-step iterative scheme from any initial point $\vec x^0\in\pi$
\begin{numcases}{}
\vec x^{k+1}=\argmin\limits_{\vec x\in T} I(\vec x)-\lambda^k \|\vec x\|_{1,d}, \label{eq:twostep_x}
\\
\lambda^{k+1}=\frac{I(\vec x^{k+1})}{\|\vec x^{k+1}\|_{1,d}},
\label{eq:twostep_lam}
\end{numcases}
produces a sequence $\lambda^k$ convergent to the global minimum of $I(\vec x)$ in $\pi$.
\end{theorem}

\begin{proof}
According to Lemma \ref{lm:x1}, $\lambda^k\ge \lambda^{k+1}\ge 0$ for any $k\in \mathbb{N}^+$, so there exists $\lambda^*\ge 0$ such that $\lim\limits_{k\to+\infty}\lambda^k=\lambda^*$.
By the definition of $T$, we know that $\vec x\in T$ implies $\frac{x}{\|\vec x\|_{1,d}}\in \pi$, and $\min\limits_{\vec x\in \pi}I(\vec x)=\min\limits_{\vec x\in T}\frac{I(\vec x)}{\|\vec x\|_{1,d}}$. Next we will prove that $\lambda^*=\min\limits_{\vec x\in T}\frac{I(\vec x)}{\|\vec x\|_{1,d}}$.

Denote $g(\lambda)=\min\limits_{\vec x\in T}I(\vec x)-\lambda\|\vec x\|_{1,d}$. It is easy to see that $g(\lambda)$ is continuous on $\mathbb{R}$ since $T$ is compact. Hence, we have $$g(\lambda^*)=\lim\limits_{k\to+\infty}g(\lambda^k)
=\lim\limits_{k\to+\infty} \|\vec x^{k+1}\|_{1,d}
(\lambda^{k+1}-\lambda^k)=0,$$
which implies that
$$
I(\vec x)-\lambda^* \|\vec x\|_{1,d} \ge 0 \Leftrightarrow I(\vec x/\|\vec x\|_{1,d}) = I(\vec x) / \|\vec x\|_{1,d} \ge \lambda^*$$
holds for all $\vec x\in T$.
\end{proof}

However, in solving the minimization problem~\eqref{eq:twostep_x}, the first difficulty is the nonlinear constraint $0\in\median(\vec x)$ (\ie $\vec x\in\pi$ by Theorem~\ref{th:median}).
Taking the complicated geometry of $\pi$ into account, one often performs the minimization in the larger space $X$, and then projects the resulting solution into $\pi$ through a simple operation: $x-\alpha \vec{1}$, where $\alpha\in \median(\vec x)$~\cite{SzlamBresson2010,HeinBuhler2010}.

The second difficulty comes from the non-differentiability of the objective function in Eq.~\eqref{eq:twostep_x}.  The following results show that explicit subgradient methods hardly work.

\begin{lemma}
Let $f$ be a Lipschitz function and $\vec x^0\in int(dom(f))$. Then the subgradient $\partial f(\vec x^0)$ is a nonempty compact convex set. Let $$\vec v^0=\argmin\{\|\vec v\|:\vec v\in  \partial f(\vec x^0)\}.$$
Then we have (1) $f(\vec x^0-t\vec v^0)\le f(\vec v^0)$ for any sufficiently small $t>0$; (2) if $0\not\in\partial f(\vec x^0)$, $f(\vec x^0-t\vec v^0)< f(\vec v^0)$ for any sufficiently small $t>0$.
\end{lemma}

\begin{theorem}\label{th:explicit_subg}
Let $\vec x^0\in \pi$ and $\lambda^0=F(\vec x^0)$, where the function $F$ is defined in~\eqref{eq:I1F}. Assume
\begin{align}
\vec v^0
&=\left(\sum_j z_{ij}^0-\lambda^0d_ic_i^0\right)_i
\nonumber \\
&=\argmin\limits_{z_{ij}+z_{ji}=0} \left\{\sum_i\left(\sum_j z_{ij}-\lambda^0d_ic_i\right)^2: z_{ij}\in \sgn(x_i-x_j), c_i\in \sgn(x_i)\right\},\label{eq:subg}
\end{align}
then we have $F(\vec x^0-t\vec v^0)\le F(\vec x^0)$ for any sufficiently small $t>0$. Moreover, if $\vec x^0$ is not a critical point of $F$, then we have $F(\vec x^0-t\vec v^0)<  F(\vec x^0)$ for any sufficiently small $t>0$.
\end{theorem}


From Theorem~\ref{th:explicit_subg}, it seems that the decreasing direction $-\vec v^0$ should be our choice if the minimization problem~\eqref{eq:subg} is solved (In fact, it is also difficult) and then we expect that $F(\vec x^0-t\vec v^0)< F(\vec x^0)$ for a steplength $t^0>0$ if $\vec x^0$ is not a critical point.
However, the direction $-\vec v^0$ may not be acceptable. The reason is twofold: First, $\vec x^1=\vec x^0-t^0\vec v^0\in\pi$ is not guaranteed; Second, after projecting $\vec x^1$ onto $\pi$, \ie $\tilde{\vec x}^1=\vec x^1-\median(\vec x^1)$,
$F(\tilde{\vec x}^1)\le F(\vec x^0)$ may not hold.
In view of this, we turn to the implicit subgradient methods. The basic requirement is to ensure $F(\tilde{\vec x}^1)\le F(\vec x^0)$. To this end, a clever relaxation introduced in~\cite{HeinBuhler2010} is used.

\begin{lemma}\label{lemma3}
Let $(\cdot,\cdot)$ denote the standard inner product in $\mathbb{R}^n$.
Then we have
(1) $(u,\vec x)=\|\vec x\|_{1,d}$, $\forall u\in\partial \|\vec x\|_{1,d}$, $\forall \vec x\in\mathbb{R}^n$;
(2) $(u,\vec y)\le\|\vec y\|_{1,d}$, $\forall u\in \partial \|\vec x\|_{1,d}$, $\forall x, \vec y \in\mathbb{R}^n$.
\end{lemma}

The term $\|\vec x\|_{1,d}$ is replaced by the inner produce $(\vec v^k,\vec x)$, where $\vec v^k\in\partial \|\vec x^k\|_{1,d}$
(this relaxation enlarges the objective function),
the two-step iterative scheme~\eqref{eq:twostep_x}-\eqref{eq:twostep_lam} turns out to be
\begin{numcases}{}
\vec x^{k+1}=\argmin\limits_{x\ne 0, 0\in\median(\vec x)} I(\vec x)-\lambda^k (\vec v^k,\vec x), \;\;\; \vec v^k\in\partial \|\vec x^k\|_{1,d}, \label{eq:twostep_x1}
\\
\lambda^{k+1}=F(\vec x^{k+1}),
\label{eq:twostep_lam1}
\end{numcases}
and then we can easily recover the four-step iterative scheme~\cite{HeinBuhler2010} as follow:
\begin{equation}\label{eq:4step}
\begin{cases}
\vec x^{k+1}=\argmin\limits_{\|\vec x\|_2\leq 1} I(\vec x)-\lambda^k (\vec v^k,\vec x), \;\;\; \vec v^k\in\partial \|\vec x^k\|_{1,d},
\\
\vec x^{k+1}=\vec x^{k+1}-\alpha\cdot \vec{1}, \alpha\in \median(\vec x^{k+1}),
\\
\text{Choose}\; \vec v^{k+1}\in\partial \|\vec x^{k+1}\|_{1,d}\; \text{such that}\; (\vec v^{k+1},\vec{1})=0,\\
\lambda^{k+1}=F(\vec x^{k+1}).
\end{cases}
\end{equation}

\begin{theorem}[Theorem 4.1 in \cite{HeinBuhler2010}]
For given $\vec x^0\in\pi$ and $\vec v^0\in\partial\|\vec x^0\|_{1,d}$ satisfying $(\vec v^0,\vec{1})=0$, $\vec x^1$ is generated by the scheme~\eqref{eq:4step}. Then we have $\vec x^1\in\pi$ and $F(\vec x^1)\le F(\vec x^0)$.
\end{theorem}

\begin{remark}~
\begin{itemize}
\item In \cite{HeinBuhler2010}, the iterative scheme~\eqref{eq:4step} is called the inverse power method, which is an extension of the inverse power method for linear systems. While a different view presented here is to regard it as some kind of relaxation introduced in Lemma~\ref{lemma3}.

\item In the scheme~\eqref{eq:4step}, it should be noted that the $2$-norm is only chosen for algorithmic convenience and does not make any effect on the final results because
the function $F(\vec x)$ is homogeneous of degree zero.

\item Furthermore, we can also add an extra positive term to make the induced optimization problem bounded from below. This raises the following generalized steepest descent (GSD) method.
\end{itemize}
\end{remark}

\begin{theorem}\label{th:gsd}
For given $\vec x^0\in\pi$ and $\vec v^0\in\partial\|\vec x^0\|_{1,d}$ satisfying $(\vec v^0,\vec{1})=0$,
let $G_{\vec x^0}(\vec x)$ be a convex function on $\mathbb{R}^n$ and $\vec x^0$ is the minimum of $G_{\vec x^0}(\vec x)$.
Assume
\begin{equation}\label{eq:gsd}
\vec x^1=\argmin\limits_{\vec x\in \mathbb{R}^n}I(\vec x)-\lambda^0(\vec v^0,\vec x)+\lambda^0 G_{\vec x^0}(\vec x),
\end{equation}
 then $F(\vec x^1)\le F(\vec x^0)$. Moreover, $F(\tilde{x}^1)\le F(\vec x^0)$ holds for $\tilde{x}^1=\vec x^1-\alpha\vec{1}$ with $\alpha\in\median(\vec x^1)$.
\end{theorem}
\begin{proof}
By the conditions, we know that $\vec x^1$ is a critical point of $I(\vec x)-\lambda^0(\vec v^0,\vec x)+\lambda^0 G_{\vec x^0}(\vec x)$, \ie
$$0\in \partial I(\vec x^1) -\lambda^0 \vec v^0+\lambda^0 \partial G_{\vec x^0}(\vec x^1),$$ so
$\lambda^0(\vec v^0-\partial G_{\vec x^0}(\vec x^1))\cap\partial I(\vec x^1)\ne\emptyset$. Let $g^1\in \partial G_{\vec x^0}(\vec x^1)$ such that $\lambda^0(\vec v^0-g^1)\in \partial I(\vec x^1) $, then we have
\begin{align*}
I(\vec x^0)-I(\vec x^1)&\ge \lambda^0(\vec v^0-g^1,\vec x^0-\vec x^1)
\\&=\lambda^0(\vec v^0,\vec x^0-\vec x^1)-\lambda^0(g^1,\vec x^0-\vec x^1)
\\&\ge
\lambda^0(\vec v^0,\vec x^0-\vec x^1)-(G_{\vec x^0}(\vec x^0)-G_{\vec x^0}(\vec x^1))
\\&=\lambda^0(\vec v^0,\vec x^0-\vec x^1)+G_{\vec x^0}(\vec x^1)-G_{\vec x^0}(\vec x^0)
\\&\ge
\lambda^0(\vec v^0,\vec x^0-\vec x^1).
\end{align*}

Then taking $\vec x^1=\vec x^1-\alpha\vec{1}$ with $\alpha\in\median(\vec x^1)$ leads to
\begin{align*}
I(\vec x^0)-I(\vec x^1)&\ge\lambda^0(\vec v^0,\vec x^0-\vec x^1)
\\&=\lambda^0
(\vec v^0,\vec x^0)-\lambda^0 (\vec v^0,\vec x^1)
\\&\ge \lambda^0 \|\vec x^0\|_{1,d}-\lambda^0
\|\vec x^1\|_{1,d}
\\&=I(\vec x^0)-\lambda^0 \|\vec x^1\|_{1,d}
\end{align*} Therefore, we get
$I(\vec x^1)-\lambda^0 \|\vec x^1\|_{1,d}\le0$, \ie
$\lambda^1=\frac{I(\vec x^1)}{\|\vec x^1\|_{1,d}}\le\lambda^0$.
\end{proof}

Theorem~\ref{th:gsd} implies that a further relaxation to enlarge the objective function can be allowed and the following iterative scheme, \ie the GSD method,  should also work,
\begin{equation}\label{eq:4step-gsd}
\begin{cases}
\vec x^{k+1}=\argmin\limits_{\vec x\in\mathbb{R}^n} I(\vec x)-\lambda^k (\vec v^k,\vec x) + \lambda^k G_{\vec x^k}(\vec x), \;\;\; \vec v^k\in\partial \|\vec x^k\|_{1,d},
\\
\vec x^{k+1}=\vec x^{k+1}-\alpha\vec{1}, \alpha\in \median(\vec x^{k+1}),
\\
\text{Choose}\; \vec v^{k+1}\in\partial \|\vec x^{k+1}\|_{1,d}\; \text{such that}\; (\vec v^{k+1},\vec{1})=0,\\
\lambda^{k+1}=F(\vec x^{k+1}).
\end{cases}
\end{equation}

\begin{remark}~
\begin{itemize}
\item  Taking $G_{\vec x^k}(\vec x)=\frac{\|x-\vec x^k\|^2_2+c^2 \|\vec v^k\|_2^2+2c(\vec x^k,\vec v^k)}{2c}$ for $c>0$,
the scheme~\eqref{eq:4step-gsd} reduces to
the SD algorithm~\cite{BressonLaurentUminskyBrecht2012}. Then the updatation of $\vec x^{k+1}$ becomes
\begin{equation}\label{eq:sd}
\vec x^{k+1}=\argmin\limits_{\vec x\in\mathbb{R}^n} I(\vec x)+\frac{\lambda^k}{2c}\|x-(\vec x^k+c\vec v^k)\|_2^2,
\end{equation}
which is nothing but the so-called proximal gradient often used in implicit subgradient methods \cite{ChambollePock2011,BressonLaurentUminskyBrecht2012,Ma2012}.
\item The relaxation from Eq.~\eqref{eq:twostep_x} to Eq.~\eqref{eq:twostep_x1} is more essential than that
from Eq.~\eqref{eq:twostep_x1} to Eq.~\eqref{eq:gsd}.
\end{itemize}
\end{remark}

In summary, by exploiting the idea of relaxing the objective function for searching the descending direction in $\pi$, adopted by either IP or SD,
we obtain two kinds of CD methods: CD1 and CD2. In CD1,
the inner problem appeared in the first line of Eq.~\eqref{eq:4step} is devised in choosing a descending direction, and then the next cell,
while the inner problem~\eqref{eq:sd} is for CD2.
The skeleton of CD methods is given in Algorithm~\ref{alg:cd}, where the CD1 method is taken as an example.

\begin{algorithm}[htpb]
\SetKwInOut{Input}{Input}
\SetKwInOut{Output}{Output}

\Input{The initial cell $\triangle^1\in\Pi$.}
\Output{The cell $\triangle^\ast$ containing the minimizer $\vec x^\ast$ with the minimum $I^\ast$, the number of traversed cells $\#_{\triangle}$.}
Set $k=1$ and $I(\vec x^0)=1$\;
\While{True}
{
Solve $\vec x^{k} = \argmin\limits_{\vec x\in\overline{\triangle^{k}}} I(\vec x)$\;
\If{$I(\vec x^{k}) > I(\vec x^{k-1})$}
{
\Return $\triangle^\ast\gets \triangle^{k-1}$, $\vec x^\ast\gets \vec x^{k-1}$, $I^\ast\gets I(\vec x^{k-1})$, $\#_\triangle \gets k$\;
 Exit the loop\;
}
Choose $\vec v^{k}\in\partial\|\vec x^{k}\|_{1,d}$ such that $(\vec v^k,\vec{1})=0$, \ie
$v_i^k =
\begin{cases}
d_i \sign(x^k_i), & \text{if } x^k_i\ne 0\\
\frac{\delta^-(\vec x^k)-\delta^+(\vec x^k)}{\delta^0(\vec x^k)}, & \text{if } x^k_i= 0
\end{cases}
$\;
Solve $\vec y^k = \argmin\limits_{\|\vec x\|_2\leq 1} I(\vec x) - F(\vec x^k)(\vec v^k,\vec x)$\;\label{line9}
Set $\vec y^k\gets \vec y^k-\median(\vec y^k)$\;\label{line}
Set $\triangle^{k+1} = \triangle_{\sign(\vec y^k)}$\;
\If{$\triangle^{k+1}\in\{\pm\triangle^1,\pm\triangle^2,\cdots,\pm\triangle^k\}$}
{
\Return $\triangle^\ast\gets \triangle^{k}$, $\vec x^\ast\gets \vec x^{k}$, $I^\ast\gets I(\vec x^{k})$, $\#_\triangle \gets k$\;
Exit the loop\;
}
Set $k \gets k+1$\;
}
\caption{\small The skeleton of cell descent algorithms.}
\label{alg:cd}
\end{algorithm}

\begin{theorem}
\label{th:cd1convergence}
The sequence $\{I(\vec x^k)\}$ produced by the CD1 method is decreasing and convergent to a local minimum. Furthermore, the sequence $\{\vec x^k\}$ produced by the CD1 method converges to an eigenvector of the 1-Laplacian with eigenvalue $\lambda^*\in [h(G), I(\vec x^0)]$.

\end{theorem}
\begin{proof}
Since $\frac{\vec y^{k-1}}{\|\vec y^{k-1}\|_{1,d}}\in \triangle^k$, we obtain
\begin{equation}
\label{eq:1}
F(\vec y^{k-1})=I\left(\frac{\vec y^{k-1}}{\|\vec y^{k-1}\|_{1,d}}\right)\ge \min\limits_{\vec x\in\overline{\triangle^k}}I(\vec x)= I(\vec x^k).
\end{equation}

For any $\alpha\in \median(\vec y^k)$,
we can readily arrive at
\begin{align*}
I(\vec y^k-\alpha\cdot \vec1)-F(\vec x^k)(\vec v^k,\vec y^k-\alpha\cdot \vec1)&=
I(\vec y^k)-F(\vec x^k)(\vec v^k,\vec y^k)
\\&\le I(\vec x^k)-F(\vec x^k)(\vec v^k,\vec x^k)
\\&=I(\vec x^k)-F(\vec x^k)\|\vec x^k\|_{1,d}=0,
\end{align*}
since $\vec y^k$ is a minimizer of $I(\vec x)-F(\vec x^k)(\vec v^k,\vec x)$ and $(\vec v^k,\vec1)=0$,
where the translation invariance of $I(\vec x)$ and  Lemma~\ref{lemma3} are applied.
Accordingly, after the updation of $\vec y^k\leftarrow \vec y^k-\alpha\cdot \vec1$ in Line~\ref{line} of Algorithm~\ref{alg:cd}, we have
\begin{equation}
\label{eq:2}
I(\vec y^k)\le F(\vec x^k)(\vec v^k,\vec y^k)\le F(\vec x^k)\|\vec y^k\|_{1,d} \Leftrightarrow F(\vec x^k)\ge {I(\vec y^k)}/{\|\vec y^k\|_{1,d}}=F(\vec y^k).
\end{equation}

Combining Eqs.~\eqref{eq:1} and \eqref{eq:2} leads to
$$
I(\vec x^0)\ge \cdots\ge F(\vec y^{k-1})\ge I(\vec x^k)\ge F(\vec y^k)\ge I(\vec x^{k+1})\ge\cdots \ge 0,
$$
where the fact that $I(\vec x^k)=F(\vec x^k)$ holds for $\vec x^k\in \overline{\triangle^k}\subset X$ has been used. That is, $I(\vec x^k)$ converges to $\lambda^*$ and
$\lambda^*\ge h(G)$ due to Theorem~\ref{th:chang2014}.
Below we will further show this $\lambda^*$ is actually an eigenvalue of $\Delta_1$.

Because there are only finite cells in $\Pi$, the decreasing sequence $\{I(\vec x^k)\}_{k=1}^\infty$ can only take finite different values. That means there exists $K\in\mathbb{Z}^+$ such that $I(\vec x^k)=F(\vec y^k)=\lambda^k=\lambda^*$ for any $k\ge K$. Furthermore, we claim that any $\vec x^k(k\ge K)$ must be the minimizer of $I(\vec x) - F(\vec x^k)(\vec v^k,\vec x)$ in the open ball $\{\vec x:\|\vec x\|_2< 1\}$. The reason is, (1) we have $\|\vec x^k\|_2^2 \leq \sum_{i=1}^n d_i |x_i^k|^2 <(\sum_{i=1}^n d_i |x_i^k|)^2= \|\vec x^k\|_{1,d}^2=1$, i.e. $\|\vec x^k\|_2 < 1$; (2) $I(\vec x^k) - F(\vec x^k)(\vec v^k,\vec x^k)=0$ holds due to Lemma~\ref{lemma3}; (3) for any $\vec x\in \{\vec x:\|\vec x\|_2< 1\}$, we have
$$
I(\vec x) - F(\vec x^k)(\vec v^k,\vec x)\ge I(\vec y^k) - F(\vec x^k)(\vec v^k,\vec y^k)\ge I(\vec y^k) - F(\vec x^k)\|\vec y^k\|_{1,d}=0,
$$
where we have used Line~\ref{line9} of Algorithm~\ref{alg:cd} and Lemma~\ref{lemma3}.
Therefore, $\vec x^k$ is a critical point of $I(\vec x) - F(\vec x^k)(\vec v^k,\vec x)$ and then
$$
0\in \partial I(\vec x) - F(\vec x^k)\partial (\vec v^k,\vec x)=\left(\sum_{j\sim i} \sgn(x_i-x_j)-\lambda^k  v^k_i\right)_{i=1}^n.
$$
Consequently, it follows from $v^k_i\in d_i\sgn(x_i)$ that $(\vec x^k,\lambda^k)$ is an eigenpair of $\Delta_1$. Hence we have $\lambda^*=\lambda^k$ is an eigenvalue.
\end{proof}

For the CD2 method, we still have a similar local convergence. In the CD1 method as shown in Algorithm~\ref{alg:cd},
the main cost is spent on two inner minimization problems. The first one is shown in Line 3 and the second one is shown in Line 9, both of which are convex and then can be solved by any standard solver for convex optimization, \eg the alternating direction method~\cite{BoydParikhChuPeleatoEckstein2011,Ma2012} and the MOSEK solver. The same situation applies to the CD2 method.

%
%
%

\section{Numerical experiments}
\label{sec:num}

Applying
the idea of relaxing the objective function in $\pi$,
a CD algorithm framework in solving the $1$-Laplacian Cheeger cut problem is proposed and two specified CD methods (\ie CD1 and CD2) are presented. The CD1 method adopts the inner problem Eq.~\eqref{eq:4step} to search the next cell,
while the CD2 method uses~\eqref{eq:sd}. Actually, we have shown that both IP and SD methods can be readily recovered from such relaxation.

\begin{figure}[htpb]
\centering
\subfigure[A Petersen graph on $10$ vertices.]{
\label{fig:typical:petersen}
\includegraphics[scale=0.35]{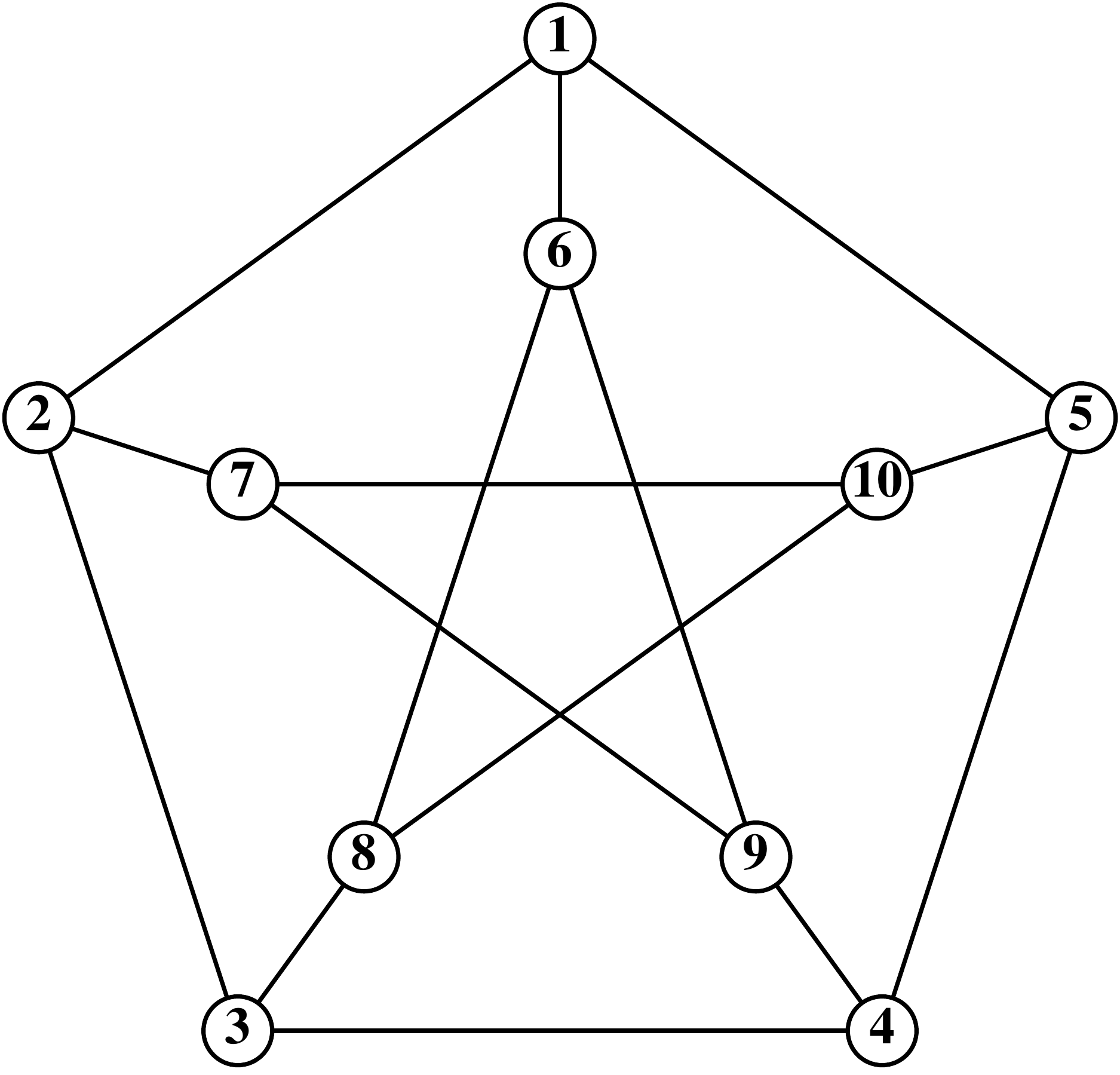}}
\subfigure[A path graph on $10$ vertices $P_{10}$.]{
\includegraphics[scale=0.65]{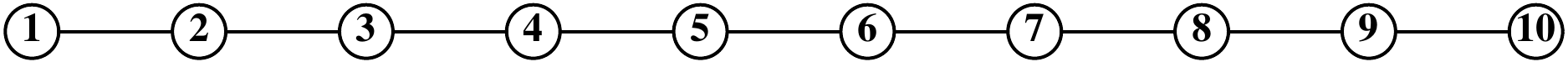}}
\subfigure[A roach Graph on $20$ vertices $R_{20}$.]{
\label{fig:typical:roach}
\includegraphics[scale=0.65]{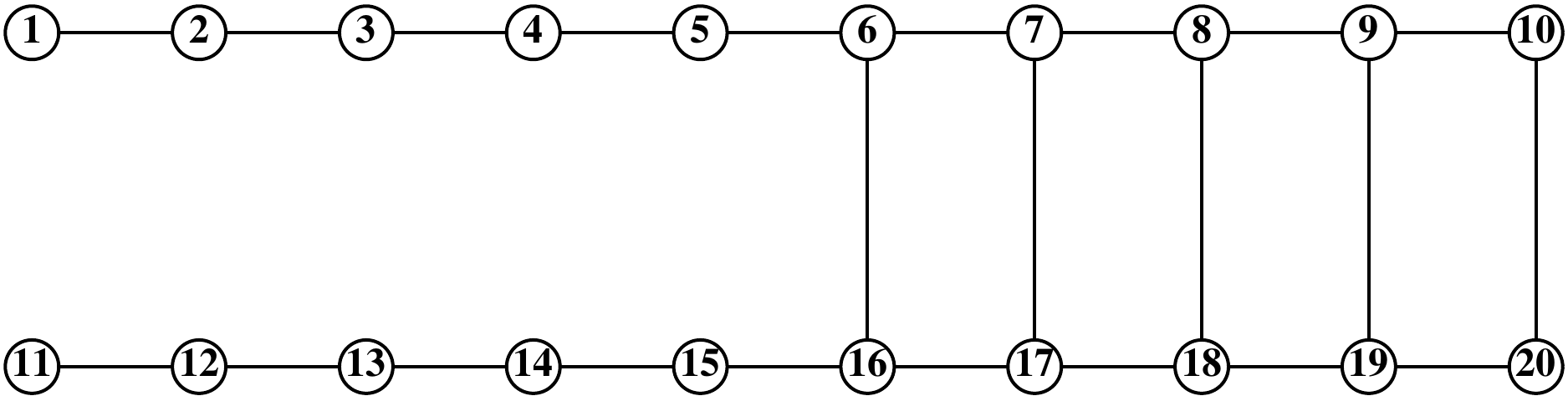}}
\caption{Several typical graphs.}
\label{fig:typical}
\end{figure}

In this section, we conduct a compare study of the CD1, CD2, IP and SD methods in terms of performance on several typical graphs, \eg the Petersen graph, the path graph (\eg $P_{10}$), the roach graph and the complete graph (\eg $K_{10}$). The Petersen graph has 10 vertices (see Fig.~\ref{fig:typical:petersen}) and the exact Cheeger value is $h(G)=\mu_2=\frac13$ while the second eigenvalue of the standard graph Laplacian is $\lambda_2=2\mu_2=\frac{2}{3}$. That is, the Cheeger value takes the lower bound of inequality (\ref{eq:mu_2&lambda_2}) on the Petersen graph, implying that the Petersen graph should be one of the most typical graphs on which the $2$-spectral clusering fails to determine the Cheeger cut.
Another interesting graph is the so-called roach graph~\cite{Luxburg2007}, which has $4k$ vertices (see Fig.~\ref{fig:typical:roach} for $k=5$, \ie a roach graph $R_{20}$) and is a typical graph that the $2$-spectral clusering fails, too. Four roach graphs of different size (\ie $R_{4k}$ with $k=2,3,4,5$) are tested.
That is, a comparison between the $1$- and $2$-Laplacian Cheeger cuts in the roach graphs makes no sense. Moreover, this work focuses on the $1$-Laplacian Cheeger cut, while its comparison with the $p$-spectral clusering ($1<p<2$) methods can be found in \cite{SzlamBresson2010,HeinBuhler2010}.

To solve all envolved convex inner problems as well as to make a fair comparison, we use the same MOSEK solver with CVX, a package for specifying and solving convex programs~\cite{cvx,gb08}. We set $c=1$ in the inner problem~\eqref{eq:sd} as in~\cite{BressonLaurentUminskyBrecht2012}.


\begin{table}
\centering
\caption{\small Performance of the IP, SD, CD1 and CD2 methods on typical graphs. Avergae number of iterations and the percent for the numerical Cheeger value $\tilde{h}$ reaching the exact Cheeger value $h$.}
\label{tab:htilde=h}
\newsavebox{\tablebox}
\begin{lrbox}{\tablebox}
\begin{tabular}{c|c|c|c|c|c|c|c|c}
\hline
\multirow{2}{*}{Graph} & \multicolumn{4}{c|}{Average No. of iterations} & \multicolumn{4}{c}{Percent for $\tilde{h}=h$}  \\
\cline{2-9}
& IP & SD & CD1 & CD2 & IP & SD & CD1 & CD2  \\
\hline
Petersen & 2.889 &	4.925	 & 2.746 &	2.789 &	33.1 & 32.4 & 37& 31 \\
\hline
$P_{10}$ &	3.984 & 7.438	&	2.774 &	3.474 &	79.8 &	76.5 &	98.6 & 86.1 \\
\hline
$K_{10}$ & 2.435 &	4.277& 3.191  &	3.128 &	61.1 & 85.3 &	100&100 \\
\hline
$R_8$ & 3.216	 &5.794 & 2.753 &2.727 &	43.2 & 44.6 & 82.4&51.7 \\\hline
$R_{12}$ & 3.481 &6.521	&	2.722 &	2.989 &	97.3 &	95.9 &	100& 98.9 \\\hline
$R_{16}$ &	3.887 &	9.957&	2.843 &	2.826 &	49.1 &	52.1 &	52.7&42.3 \\
\hline
$R_{20}$ & 4.239&10.782	&	3.087  &	3.02 &	52.7 &	55.7 &	57.7&47.9\\
\hline
\end{tabular}
\end{lrbox}
\scalebox{0.8}{\usebox{\tablebox}}
\end{table}

\begin{table}
\centering
\caption{\small Comparison of the numerical Cheeger values.}
\label{tab:compare}
\begin{lrbox}{\tablebox}
\begin{tabular}{c|c|c|c|c|c|c}
\hline
\multirow{3}{*}{Graph} & \multicolumn{3}{c|}{IP \vs CD1} & \multicolumn{3}{c}{SD \vs CD2}\\
\cline{2-7}
& \multicolumn{3}{c|}{Percent for $\tilde{h}_{IP}\lesseqqgtr \tilde{h}_{CD1}$}
& \multicolumn{3}{c}{Percent for $\tilde{h}_{SD}\lesseqqgtr \tilde{h}_{CD2}$} \\
\cline{2-7}
& $\tilde{h}_{IP} < \tilde{h}_{CD1}$ & $\tilde{h}_{IP} > \tilde{h}_{CD1}$ & $\tilde{h}_{IP} = \tilde{h}_{CD1}$ &
$\tilde{h}_{SD} < \tilde{h}_{CD2}$ & $\tilde{h}_{SD} > \tilde{h}_{CD2}$ & $\tilde{h}_{SD} = \tilde{h}_{CD2}$
\\
\hline
Petersen & 0.8 &	11& 88.2 &	7.6 &	7.5 & 84.9  \\
\hline
$P_{10}$ &	0.4&	19.2&	80.4&	8&	18.8&	81.4	 \\
\hline
$K_{10}$ & 0&	38.9&	61.1&	0&	14.7& 85.3	 \\
\hline
$R_8$ & 0.5	&39.8&	59.7&	13.7&	28.2&	58.1 \\ \hline
$R_{12}$ & 0&	2.7&	97.3&	0.8&	3.8& 95.4\\ \hline
$R_{16}$ & 9.1&	14.4&	76.5&	20.1&	11.90&68	\\ \hline
$R_{20}$ & 10.3&	16&	73.7&	19.9&	13.3& 66.8	\\ \hline\hline
\multirow{3}{*}{Graph} & \multicolumn{3}{c|}{SD \vs CD1} & \multicolumn{3}{c}{CD1 \vs CD2}\\
\cline{2-7}
& \multicolumn{3}{c|}{Percent for $\tilde{h}_{SD}\lesseqqgtr \tilde{h}_{CD1}$}
& \multicolumn{3}{c}{Percent for $\tilde{h}_{CD1}\lesseqqgtr \tilde{h}_{CD2}$} \\
\cline{2-7}
& $\tilde{h}_{SD} < \tilde{h}_{CD1}$ & $\tilde{h}_{SD} > \tilde{h}_{CD1}$ & $\tilde{h}_{SD} = \tilde{h}_{CD1}$ &
$\tilde{h}_{CD1} < \tilde{h}_{CD2}$ & $\tilde{h}_{CD1} > \tilde{h}_{CD2}$ & $\tilde{h}_{CD1} = \tilde{h}_{CD2}$
\\
\hline
Petersen & 7.3 &9& 83.7  &	8.4 &	6.4 & 85.2  \\
\hline
$P_{10}$ &	0.4& 22.5&	77.1 &	12.5&	0&	87.5	 \\
\hline
$K_{10}$ & 0&	14.7& 85.3	 &	0&	0&	100 \\
\hline
$R_8$ & 2.9	&40.5& 56.6 & 	32.7& 2.2&	65.1 \\ \hline
$R_{12}$ &  0&	4.1& 95.9	 &	1.1  &	0& 98.9\\ \hline
$R_{16}$ & 12.2&	15.7&	72.1&	13.8&	1.8&	84.4\\ \hline
$R_{20}$ & 11.7&15.1&	73.2&	14.4&	4&81.6	\\
\hline\hline
\multirow{3}{*}{Graph} & \multicolumn{3}{c|}{SD \vs IP} & \multicolumn{3}{c}{IP \vs CD2}\\
\cline{2-7}
& \multicolumn{3}{c|}{Percent for $\tilde{h}_{SD}\lesseqqgtr \tilde{h}_{IP}$}
& \multicolumn{3}{c}{Percent for $\tilde{h}_{IP}\lesseqqgtr \tilde{h}_{CD2}$} \\
\cline{2-7}
& $\tilde{h}_{SD} < \tilde{h}_{IP}$ & $\tilde{h}_{SD} > \tilde{h}_{IP}$ & $\tilde{h}_{SD} = \tilde{h}_{IP}$ &
$\tilde{h}_{IP} < \tilde{h}_{CD2}$ & $\tilde{h}_{IP} > \tilde{h}_{CD2}$ & $\tilde{h}_{IP} = \tilde{h}_{CD2}$
\\
\hline
Petersen & 12.4 &3.6& 84  &	5.9 &	14.6 & 79.5  \\
\hline
$P_{10}$ &	11.4 & 15&	 73.6 &	8.8&	16.2&	75	 \\
\hline
$K_{10}$ & 30.2&	6&  63.4	 &	0&	38.9&	61.1 \\
\hline
$R_8$ &  13.7	&12.2&  74.1 & 	10.2& 19&	70.8 \\ \hline
$R_{12}$ &  0.6&	2& 97.4	 &	0.8  &	2.4& 96.8\\ \hline
$R_{16}$ & 10.6&	8.6&	80.8&	17.4&	10.7&	71.9\\ \hline
$R_{20}$ & 10&8.1&	81.9&	16.7&	12.5&70.8	\\ \hline
\end{tabular}
\end{lrbox}
\scalebox{0.8}{\usebox{\tablebox}}
\end{table}

During all the experiments,
we start from the same $1000$ initial random cuts and record the numbers of iterations as well as the resulting numerical Cheeger values. Let $\tilde{h}$ denote the numerical Cheeger value. Table~\ref{tab:htilde=h} shows the percent for  $\tilde{h}$ equalling the exact Cheeger value $h$ ($h$ equals to $\frac13,\frac19,\frac59,\frac13,\frac15,\frac18,\frac{1}{11}$ for the Petersen graph, $P_{10}$, $K_{10}$, $R_8$, $R_{12}$, $R_{16}$ and $R_{20}$, respectively) and the average number of iterations required in those four methods. From there
it is easily seen that: (a) the SD method needs more iterations than the other and its average number of iterations increases quickly on the roach graph as the graph size increases; (b) both CD1 and CD2 methods require almost the same average numbers of iterations on all seven graphs, which are also less than those for the IP method; (c) the CD1 method has the highest percent to reach the exact Cheeger value.
Table~\ref{tab:compare} further presents
the pairwise comparison of the numerical Cheeger values obtained by the IP, SD, CD1 and CD2 methods.
Overall, the CD1 method provides the best cut among those four methods.
For example, on the Petersen graph, the percent for $\tilde{h}_{IP}<\tilde{h}_{CD1}$ is far less than that for $\tilde{h}_{IP}>\tilde{h}_{CD1}$, which means, within the 1000 numerical Cheeger values from the same $1000$ initial random cuts,
the CD1 method produces much more smaller numerical Cheeger values than the IP method. In such sense, we have the order on the Petersen graph: CD1 $>$ SD $\approx$ CD2 $>$ IP. In summary, the CD1 method has the better performance in terms of both the number of iterations and the quality of cut, but the CD2 method does not show clear advantage in the quality of cut.
That is, while using the the CD framework to design a specified algorithm, as we expected, the strategy for searching next cell plays a key role. The numerical results here show that the strategy using the inner problem in Eq.~\eqref{eq:4step} is better than that using the inner problem~\eqref{eq:sd}.

%
%
%
%
%
%
%
%
%
%


\section{Conclusions and discussions}
\label{sec:con}

In this paper, we present a systematic review of the theory of
the $1$-Laplacian Cheeger cut, where the spectrum of the $1$-Laplacian $\Delta_1$, the property of the feasible set $\pi$ and the conncetion between the Cheeger value $h(G)$ and the second eigenvalue $\mu_2$ of $\Delta_1$ are studied in details. Taking advantage of the cell structure of $\pi$, we propose a cell descend (CD) algorithm framework for achieving the Cheeger cut. Combined with
the relaxation to guarantee
the descrease of the objective value in $\pi$, we obtain two specified CD methods. Moreover, we also show that the inverse power (IP) method and the steepest descent (SD) method can also be generally recovered from such relaxation. A compare study of all these methods are conducted on several typical graphs. We find that,  the CD1 method, which adopts the inner problem in Eq.~\eqref{eq:4step} to search the next cell, performs better than both IP and SD methods.
We must admit that all the numerical results presented in this work are preliminary since the size of the graphs is small, but precisely because of small size of those typical graphs the exact cuts can be obtained conveniently for comparison. The performance of our CD algorithm framework on graphs with relatively large size is still going on. Furthermore, developping theory and algorithms for multi-class ($\geq 3$) cut and designing other strategies than the relaxation for searching next cell will also be our furture projects.

\section*{Acknowledgement}
This research was supported by grants from the National Natural Science Foundation of China (Nos. 11371038, 11471025, 11421101, 61121002).

\section*{Appendix A: An example}


The critical points of $\frac{I(\cdot)}{\|\cdot\|_{1,d}}$ in the sense of the Clarke derivative \cite{HeinBuhler2010} are actually not equivalent to the eigenvectors of $\Delta_1$. In this Appendix, we give an example. Let $G=(V,E)$ be the following graph.
\begin{center}
\includegraphics{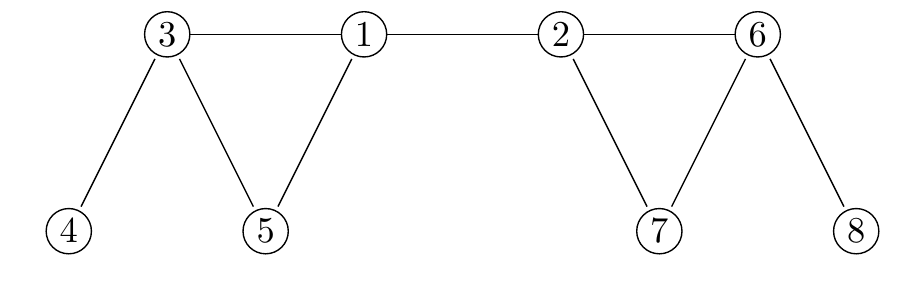}
\end{center}
Consider the point
$$
\vec x_0=\frac{1}{12}(0,0,1,1,1,-1,-1,-1) \in X.
$$
We can easily check that $\vec x_0$ is an eigenvector of $\Delta_1$ and the corresponding eigenvalue is $\mu=\frac13$. However,
for the direction
$$
\vec h_0=(1,-1,0,0,0,0,0,0)\in \mathbb{R}^8,
$$
we have
\begin{equation*}\label{eq:main1}
\limsup\limits_{\vec w\to \vec 0,t\to 0^+}
\frac{1}{t}\left(\frac{I(\vec x_0+ \vec w+t \vec h)}{\|\vec x_0+\vec w+t \vec h\|_{1,d}}-\frac{I(\vec x_0+\vec w)}{\|\vec x_0+\vec w\|_{1,d}}\right)\le -2.
\end{equation*}
That is, $\vec x_0$ is not a critical point $\frac{I(\cdot)}{\|\cdot\|_{1,d}}$ in the sense of the Clarke derivative.

\end{document}